\documentclass[12pt]{amsart}
\usepackage{graphicx}
\usepackage[T1]{fontenc}
\usepackage[utf8]{inputenc}
\usepackage{amssymb,amsmath,amsfonts, amsthm}
\usepackage{tikz}
\usetikzlibrary[calc,patterns] 
\usepackage{pinlabel}

\newtheorem{theorem}{Theorem}[section]

\newtheorem{lemma}[theorem]{Lemma}

\theoremstyle{definition}
\newtheorem{definition}[theorem]{Definition}

\theoremstyle{remark}

\title{The automorphism group of Rauzy diagrams}
\author{Corentin Boissy}

\address{Institut de Math\'ematiques de Toulouse \\ 
Universit\'e Toulouse 3 \\
118 route de Narbonne \\
31062 Toulouse, France
}
\email{corentin.boissy@math.univ-toulouse.fr}

\begin{document}
\begin{abstract}
We give a description of  the automorphism group of a Rauzy diagram as a subgroup of the symmetric group. This is based on an example that appear in some personnal notes of Yoccoz that are to be published in the project ``Yoccoz archives''.
\end{abstract}

\maketitle
\section{Introduction}
Rauzy induction was introduced in \cite{Rauzy} as a tool to study interval exchange maps. It is a renormalization process that associates to an interval exchange map, another one obtained as a first return map on a well chosen subinterval.
After the major work of Veech~\cite{Veech82}, the Rauzy induction became a powerful tool to study the Teichmüller geodesic flow. 

A slightly different tool was used by Kerckhoff \cite{Kerckhoff}, Bufetov \cite{Buf}, and Marmi-Moussa-Yoccoz~\cite{MMY}. 
It is obtained after labeling the intervals, and keeping track of them during the renormalization process. This small change was a significative improvement, and was used more recently to prove other important results about the Teichmüller geodesic flow, for instance the simplicity of the Liapunov exponents (Avila-Viana \cite{Avila:Viana}), or the exponential decay of correlations (Avila-Gouezel-Yoccoz \cite{AGY}). See also \cite{Gutier}.

This labeling induces a non trivial automorphism group on the Rauzy diagram $\mathcal{D}$, that corresponds to relabellings that preserves $\mathcal{D}$. The computation of the cardinal of $Aut(\mathcal{D})$ was done by the author in \cite{B:labeled}, by considering a geometric interpretation in terms of a moduli space of \emph{labeled} translation surfaces. However, the precise description of the elements of $Aut(\mathcal{D})$, as permutation elements was not done. In some personnal notes (annoted by the author) that are to be published in the project ``Yoccoz archives'',  Yoccoz \cite{Yoc} gives an extensive description of many ``small'' Rauzy classes, including the one corresponding to the stratum $\mathcal{H}(1,1,1,1)$ (see Section~20 of \cite{Yoc}). In particular, we can find for such Rauzy class a nice description of the automorphism group. Based on this example, we propose a similar description for the automorphism group of any (non hyperellliptic) Rauzy class.

\subsection*{Acknowledgements}
The author thanks Carlos Matheus and Pascal Hubert for giving me the text of Yoccoz that started this project and  encouraging me to write this paper. The project was partially supported by the labex CIMI project ``GEOMET''.

\section{Background}
\subsection{Combinatorial definition}
\begin{definition}
Let $\mathcal{A}$ be a finite alphabet that consists of $d$ elements. A \emph{labeled permutation} is a pair $\pi=(\pi_t,\pi_b)$ of one-to-one maps from $\mathcal{A}$ to $\{1,\dots ,d\}$. We usually represent $\pi$ by a table:
\begin{eqnarray*}
{\pi}= \left(\begin{array}{ccccc}\pi_t^{-1}(1)&\pi_t^{-1}(2)&\ldots&\pi_t^{-1}(d)
  \\
\pi_b^{-1}(1)&\pi_b^{-1}(2)&\ldots&\pi_b^{-1}(d) \end{array}\right).
\end{eqnarray*}
\end{definition}

A \emph{renumbering} of a labeled permutation is the composition of $(\pi_t,\pi_b)$ by a one-to-one map $f$ from $\mathcal{A}$ to $\mathcal{A}'$. It just corresponds to changing the labels without changing the underlying permutation $\pi_t\circ \pi_b^{-1}$. 

A labeled permutation is irreducible if for any $k\in \{1,\dots , d-1\}$, $\pi_t^{-1}\{1,\dots ,k\}\neq \pi_b^{-1}\{1,\dots ,k\}$.

We define the following maps on the set of irreducible permutations, called the combinatorial Rauzy moves.

\begin{enumerate}
\item {$\mathcal{R}_t$}: let $k=\pi_b(\pi_t^{-1}(d))$ with $k\leq d-1$. 
Then, $\mathcal R_t(\pi_t,\pi_b)=(\pi_t',\pi_b')$ where $\pi_t=\pi_t'$ and
$$
\pi_b'^{-1}(j) = \left\{
\begin{array}{ll}
\pi_b^{-1}(j) & \textrm{if $j\leq k$}\\
\pi_b^{-1}(d) & \textrm{if $j = k+1$}\\
\pi_b^{-1}(j-1) & \textrm{otherwise.}
\end{array} \right.
$$

\item {$\mathcal{R}_b$}: let $k=\pi_t(\pi_b^{-1}(d))$ with $k \leq
d-1$. Then, $\mathcal R_b(\pi_t,\pi_b)=(\pi_t',\pi_b')$ where $\pi_b=\pi_b'$ and
$$
\pi_t'^{-1}(j) = \left\{
\begin{array}{ll}
\pi_t^{-1}(j) & \textrm{if $j\leq k$}\\
\pi_t^{-1}(d) & \textrm{if $j = k+1$}\\
\pi_t^{-1}(j-1) & \textrm{otherwise.}
\end{array} \right.
$$
\end{enumerate}

\begin{definition}
A \emph{Rauzy class}, is a minimal  set of labeled  permutations invariant by the combinatorial  Rauzy moves. 

A \emph{Rauzy diagram},  is a graph whose vertices are the elements of a Rauzy class and whose vertices are the combinatorial Rauzy moves.

An automorphism of a Rauzy diagram $\mathcal{D}$ with vertices $\mathcal{R}$ is 
a graph automorphism of $\mathcal{D}$ that sends a ``$t$'' edge (resp. a ``$b$'' edge) to a ``$t$'' edge (resp. a ``$b$'' edge). We can show that it is just given by a renumbering of the vertices, \emph{i.e}, an element of of the symmetric group $\mathfrak{S}(\mathcal{A})$
\end{definition}

\subsection{Links to translation surfaces}\label{link:to:transl}
A \emph{translation surface} is a (real, compact, connected) genus $g$ surface $S $ with a translation atlas \emph{i.e.} a triple $(S,\mathcal U,\Sigma)$ such that $\Sigma$ is a finite subset of $S$ (whose elements are called {\em singularities}) and $\mathcal U = \{(U_i,z_i)\}$ is an atlas of $S \setminus \Sigma$ whose transition maps are translations of $\mathbb{C}\simeq \mathbb{R}^2$. We will require that  for each $s\in \Sigma$, there is a neighborhood of $s$ isometric to a Euclidean cone whose total angle is a multiple of $2\pi$.  One can show that the holomorphic structure on $S\setminus \Sigma$ extends to $S$ and that the holomorphic 1-form $\omega=dz_i$ extends to a holomorphic $1-$form on $X$ where  $\Sigma$ corresponds to the zeroes of $\omega$ and maybe some marked points. We usually call $\omega$ an \emph{Abelian differential}.  A zero of $\omega$ of order $k$ corresponds to a singularity of angle $(k+1)2\pi$.

For $g \geq 1$, we define the moduli space of Abelian differentials $\mathcal{H}_g$ as the moduli space of pairs $(X,\omega)$ where $X$ is a genus $g$ (compact, connected) Riemann surface and $\omega$ non-zero holomorphic $1-$form defined on $X$. The term moduli space means that we identify the points $(X,\omega)$ and $(X',\omega')$ if there exists an analytic isomorphism $f:X \rightarrow X'$ such that 
$f^*\omega'=\omega$.

One can also see a translation surface obtained as a polygon (or a finite union of polygons) whose sides come by pairs, and for each pair, the corresponding segments are parallel and of the same length. These parallel sides are glued together by translation and we assume that this identification preserves the natural orientation of the polygons. In this context, two translation surfaces are identified in the moduli space of Abelian differentials if and only if the corresponding polygons can be obtained from each other by cutting and gluing and preserving the identifications. 

The moduli space of Abelian differentials is stratified by the  combinatorics of the zeroes; we will denote by $\mathcal{H}(k_1^{n_1},\ldots ,k_r^{n_r})$ the stratum of $\mathcal{H}_g$ consisting of (classes of) pairs $(X,\omega)$ such that $\omega$ has exactly $n_1$ zeroes of order $k_1$, etc\dots   It is well known that  this space is  (Hausdorff) complex analytic.

There is a natural construction, the Veech construction (or zippered rectangle construction),  that associate to an element $\pi \in \mathcal{R}$, and a continuous datum $\zeta\in \mathbb{C}^d$, satisfying the ``suspension data condition'', a translation surface $S(\pi,\zeta)$ (see for instance \cite{MMY}, or \cite{B:labeled} for details). See Figure~\ref{complab}.
Different choices of parameter $\zeta$ give surfaces in the same connected component of a stratum in the moduli space of Abelian differentials since the set of suspension data is connex (in fact, convex).

The Veech construction, with the associate Rauzy--Veech induction defines three natural invariants of a Rauzy class:
\begin{enumerate}
\item The set with multiplicities of degree of the conical singularities of $S(\pi,*)$, \emph{i.e.} the stratum where $S(\pi,*)$ is.
\item When such stratum is non-connected, the corresponding connected component .
\item The degree of the singularity attached on the left in the construction (denoted as the \emph{special singularity}).
\end{enumerate}
It is proven in \cite{B:RC} that if two labeled permutations share the same above invariant, then up to a renumbering, they are in the same Rauzy class.

In order to have a complete classification of Rauzy classes, one needs a refinement of the invariant $(1)$. From \cite{B:labeled}, the letters in the alphabet induce a marking on the set of horizontal separatrices of each singularities, where each (outgoing) horizontal separatrix is marked by a letter, and the one corresponding to the interval in the Veech construction is marked twice (by $\pi_t^{-1}(1)$ and $\pi_b^{-1}(1)$). 

\begin{figure}[htb]
\begin{center}
\labellist
\tiny 
\hair 2pt
\pinlabel $AB$ at 10 43
\pinlabel $C$ at 55 58
\pinlabel $D$ at 125 65
\normalsize
\pinlabel $\zeta_A$ at 16 62 
\pinlabel $\zeta_B$ at 40 58
\pinlabel $\zeta_C$ at 80 60
\pinlabel $\zeta_D$ at 150 58
\pinlabel $\zeta_B$ at 8 25 
\pinlabel $\zeta_D$ at 50 5
\pinlabel $\zeta_C$ at 120 0
\pinlabel $\zeta_A$ at 165 20
\endlabellist

\includegraphics[scale=1.4]{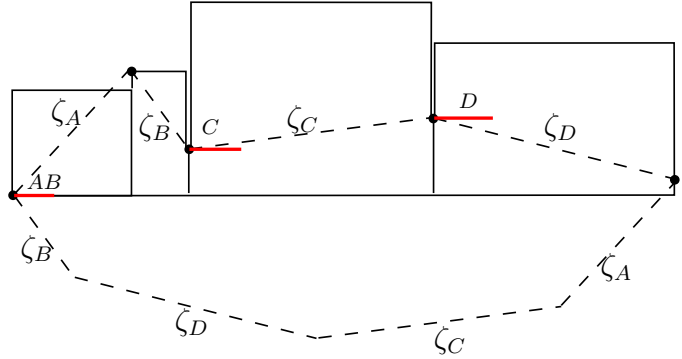}
\caption{A framing of a surface issued from the Veech construction}

\label{complab}
\end{center}
\end{figure}

Considering the corresponding moduli space $\mathcal{H}^{lab}$ of \emph{labeled} translation surfaces (\emph{i.e.} translation surfaces with such markings), we see that different choices of parameters $\zeta$ define surface in the same connected component of $\mathcal{H}^{lab}$.  In \cite{B:labeled}, it is proven that two labeled permutations are in the same Rauzy class if and only if:
\begin{enumerate}
\item The letters on the top left and bottom left are the same (they will be denoted by $\pm \infty $ as in \cite{Yoc}).
\item The canonical cyclic order on the set of labels obtained by rotating clockwise around a singularity must be the same (for one, hence any choice of surfaces built from the labeled permutations).
\item The resulting labeled translation surfaces are in the same connected component of $\mathcal{H}^{lab}$.
\end{enumerate}

Furthermore, once the cyclic order on $\mathcal{A}$  and the underlying connected component of the (nonlabeled) translation surface are fixed, we have (see \cite{B:labeled}, Theorem~1.1 and Theorem~1.3):
\begin{enumerate}
\item If there are odd degree singularity, then there exists two such Rauzy classes.
\item Otherwise there is only one such Rauzy class.
\end{enumerate}

\subsection{A surgery on translation surfaces}\label{def:breaking}
We describe the surgery called ``breaking up a singularity'',  introduced by Kontsevich and Zorich in \cite{KoZo}. 
We start from a zero $P$  of degree $k_1+k_2$. The neighborhood $V_\varepsilon=\{x\in X,  d(x,P)\leq \varepsilon\}$ of this conical singularity is obtained by considering $2(k_1+k_2)+2$ Euclidean half disks of radii $\varepsilon$  and gluing each half side of them to another one in a cyclic order. We can break the zero of order $k_1+k_2$ into a pair of singularities of order $k_1,k_2$  by changing continuously the way they are glued to each other as in Figure~\ref{break:zero}. Note that in this surgery, the metric is not modified outside $V_\varepsilon$. In particular, the boundary $\partial V_\varepsilon$   is isometric to (a connected  covering) of an Euclidean circle. Note that in this construction, we can ``rotate'' the two singularities by an angle $\theta$ by cutting the surface along $\partial V_\varepsilon$, rotating $V_\varepsilon$ by an angle $\theta$ and regluing it.

\begin{figure}[htb]
\begin{center}
\labellist
\tiny 
\hair 2pt
\pinlabel $\varepsilon$ at 35 55
\pinlabel $\varepsilon$ at 35 80
\pinlabel $\varepsilon$ at 100 55
\pinlabel $\varepsilon$ at 100 80
\pinlabel $\varepsilon$ at 60 20
\pinlabel $\varepsilon$ at 60 100

\pinlabel $\varepsilon-\delta$ at 320 55
\pinlabel $\varepsilon-\delta$ at 320 80
\pinlabel $\varepsilon-\delta$ at 385 55
\pinlabel $\varepsilon-\delta$ at 385 85
\pinlabel $\varepsilon+\delta$ at 370 20
\pinlabel $\varepsilon+\delta$ at 370 115
\pinlabel $6\pi$  at 120 0
\pinlabel $4\pi+4\pi$  at 400 0
\pinlabel $\partial V_\varepsilon$ at 280 145 
\endlabellist
\includegraphics[scale=0.6]{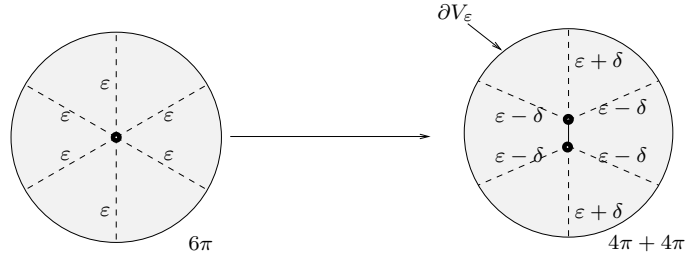}
\caption{Local surgery that break a zero of degree $k_1+k_2$ into two zeroes of degree $k_1$ and $k_2$ respectively.}
\label{break:zero}
\end{center}
\end{figure}

\subsection{Yoccoz's example}
We start with a geometrical interpretation of Yoccoz's example (see \cite{Yoc} Section 20.1). 
We consider the alphabet $\mathcal{A}=\{\pm \infty , 0, a_1,a_2,b_1,b_2, c_1,c_2\}$ and the following element:
$$
 \pi=\left ( \begin{array}{rllllllll} -\infty&b_2&a_2&b_1&a_1&c_1&0&c_2&\infty\\\infty&b_1&a_2&b_2&a_1&c_2&0&c_1&-\infty \end{array}\right )
$$
 The Veech construction with labels creates a translation surface in the stratum $\mathcal{H}(1,1,1,1)$. Each singularity has angle $4\pi$, hence gives two labeled separatrices where: 
\begin{itemize}
\item The doubly labeled separatrix is $\pm \infty $, and the other separatrix of the corresponding singularity is labeled ``0''.
\item The other three pairs of marked separatrices are $\{a_1,a_2\}$, $\{b_1,b_2\}$, and $\{c_1,c_2\}$.
\end{itemize}

As described in \cite{Yoc}, 
the subgroup $\mathcal{G}'$ of $\mathfrak{S}(\mathcal{A})$ preserving these pairings has order $2^3\times 3!=48$. The automorphism group of the Rauzy diagram $\mathcal{G}$ has order $\frac{48}{2}=24$ and is a subgroup of $\mathcal{G}'$. Hence we want to construct an morphism $\phi$ from $\mathcal{G}'$ to $\{\pm 1\}$ such that $\mathcal{G}=\ker(\phi)$.

There is a natural map from $\mathcal{G}'$ onto $\mathfrak{S}(a,b,c)$, whose kernel $N$ identifies with $\{-1,1\}^3$. For $\tau\in \mathfrak{S}(a,b,c)$, we define $\sigma(\tau)\in \mathcal{G}'$ by $\sigma(\tau)(\alpha_i)=\tau(\alpha)_i$, for $\alpha\in \{a,b,c\}$. Hence the extension $1\rightarrow N\rightarrow \mathcal{G}' \rightarrow \mathfrak{S}(a,b,c) \rightarrow 1$ is a split extension. 
Therefore, the group $\mathcal{G}'$ has a natural semi-direct structure $\mathcal{G}'= N \rtimes H\sim \{\pm 1\}^3 \rtimes \mathfrak{S}(a,b,c)$ with $H=\sigma(\mathfrak{S}(a,b,c))$.

{\bf Claim}: Under the above identification, $\mathcal{G}$ is the kernel of the map $f:((\varepsilon_a,\varepsilon_b,\varepsilon_c), \tau)\to \varepsilon_a\varepsilon_b\varepsilon_c sgn(\tau)$, where $sgn$ is the signature.

To prove the claim, we observe that:
\begin{itemize}
\item $f(\varepsilon, \tau)=f_1(\varepsilon)f_2(\tau)$, with $f_1$ a morphism from $\{\pm 1\}^3$ to $\{\pm 1\}$, and $f_2$ is a morphism from $\mathfrak{S}(a,b,c)$ to $\{\pm 1\}$, hence is either identity or the signature.
\item The elements $((-1,1,1),Id)$, $((1,-1,1),Id)$ and $((1,1,-1),Id)$ are \emph{not} in $\mathcal{G}$ from \cite{B:labeled}. Hence, $f_1(\varepsilon_a,\varepsilon_b,\varepsilon_c)=\varepsilon_a\varepsilon_b\varepsilon_c$. 
\end{itemize}
It remains to prove that $f_2\neq Id$. In this particular example (see the author's comment in \cite{Yoc}), this can be easily proven by exhibiting a particuliar path in the Rauzy diagram. 

In Lemma~\ref{lemme:geometrique}, we use a geometric construction to get an element $g=((\varepsilon_a,\varepsilon_b,1),(a,b))\in \mathcal{G}$ (hence with $\varepsilon_a\varepsilon_b f_2((a,b))=1$) such that $g^2=((-1,-1,1),Id)\neq 1_\mathcal{G}$. But a computation gives  $g^2=((\varepsilon_a\varepsilon_b,\varepsilon_a\varepsilon_b,1),Id))$. This implies $f_2((a,b))=-1$, hence $f_2=sgn$.

\section{General case}
 Let $\mathcal{D}$ be a (nonhyperelliptic) Rauzy diagram whose associate stratum is  $\mathcal{H}(k,k_1^{n_1},\dots ,k_{r}^{n_r})$, and we assume that the degree of the special (left) singularity is $k$. From  \cite{B:labeled} the order of the group $Aut(\mathcal{D})$ is:
$$\varepsilon { \prod_{i=1}^r n_i!(k_i+1)^{n_i}}$$
where $\varepsilon=1$ if all $k_i$ are even and $\frac{1}{2}$ otherwise.
There is a one-to-one map $T:\mathcal{A}\to \mathcal{A}$, such that $T(-\infty )=+\infty $, and otherwise rotating counterclockwise the separatrix $I_\alpha$ by $2\pi$ gives the separatrix $I_{T(\alpha)}$.

We consider the orbits in $\mathcal{A}$ for the action of $T$. Denote by \emph{special orbit} the one containing $\pm \infty $n and by regular orbits the other ones.For $i\in \{1,\dots ,r\}$ there are $n_i$ regular orbits of length $k_i+1$. We denote $\theta_i$ the set of orbits of length $k_i+1$, 
$\Theta_i=\{ \Theta_{i,j}, \ j\in \{1,\dots ,n_i\}$, and we chose for each $(i,j)$ an element $\alpha_{i,j}\in \Theta_{i,j}$. 

An element $\sigma\in Aut(\mathcal{D})\subset \mathfrak{S}(\mathcal{A})$   satisfies $\sigma\circ T=T\circ\sigma$. In particular, $\sigma$ induces a permutation on the set of (regular) orbits of the same size, and $\sigma$ fixes all the elements of the special orbit.

This property defines a subgroup $\mathcal{G}'$ of $\mathfrak{S}(\mathcal{A})$ of order $\prod_{i=1}^r n_i!(k_i+1)^{n_i}$. If all singularities are of even order, then $\mathcal{G}'=Aut(\mathcal{D})$ and there is nothing to do. If there are singularities of odd order, then $\mathcal{G}=Aut(\mathcal{D}$ is a subgroup of order 2 of $\mathcal{G}'$. We will define a nontrivial homomorphism $\phi$ for which $\mathcal{G}$ is the kernel.

First, we observe that $\mathcal{G}'=\prod_{i=1}^r \mathcal{G}'_i$ where $\mathcal{G}'_i$ are elements of $\mathcal{G}'$ whose support are in $Supp(\Theta_i)=\sqcup_{j} \Theta_{i,j}$. 

For each $i\in \{1,\dots ,r\}$, there is a natural map from $\mathcal{G}_i'$ onto  the group $\mathfrak{S}(\{1,\dots ,n_i\})$, whose kernel $N_i$ identifies with $U_{k_i+1}^{n_i}$, where $U_{p}$ denotes the group of complex $p-$roots of unity. For $\tau\in \mathfrak{S}(\{1,\dots ,n_i\})$, we define $\sigma_i(\tau)\in \mathcal{G}'_i$ by $\sigma_i(\tau)(\alpha_{i,j})=\alpha_{i,\tau(j)}$, and extended on $Supp(\Theta_i)$ by using $T$. Hence the extension $1\rightarrow N_i\rightarrow \mathcal{G}_i' \rightarrow \mathfrak{S}(\{1,\dots ,n_i\}) \rightarrow 1$ is a split extension. 
Therefore, the group $\mathcal{G}_i'$ has a natural semi-direct structure $\mathcal{G}_i'= N_i \rtimes H_i\sim U_{k_i+1}^{n_i}\rtimes \mathfrak{S}(\{1,\dots ,n_i\})$ with $H_i=\sigma(\mathfrak{S}(\{1,\dots ,n_i\}))$.

In order to simplify notations, we identify $\mathcal{G}_i'$ with the image of the canonical homomorphism $\mathcal{G}_i'\to \mathcal{G}'$.

Now the group $\mathcal{G}$ is a subgroup of index 2 of $ \mathcal{G}'$. Hence, is the kernel of a homomorphism $\phi:\mathcal{G} \to \{\pm 1\}$, that we write, with a slight abuse of notation as $\prod_i \phi_i$, where $\phi_i: \mathcal{G}_i' \to \{\pm 1\}$ is a homomorphism.

\begin{theorem}
Under the above identification, we have:
\begin{itemize}
\item If $k_i$ is even, then $\phi_i=Id$.
\item Otherwise, let $2p_i=k_i+1$, then:
$$\phi_i((\zeta_1,\dots,\zeta_{n_i}),\tau_i)=(\zeta_1\dots \zeta_{n_i})^{p_i} sgn(\tau_i).$$
where $sgn$ is the signature.
\end{itemize}
\end{theorem}

\begin{proof}
Let $\zeta=\exp(\frac{2i \pi}{2p_i})$. 
From the proof of Proposition~4.2 in \cite{B:labeled} (see also \cite{B:markedpoles}, Section~5.1), the element $((1,\dots ,1,\zeta,1,\dots ,1),Id)$ is \emph{not} in $\mathcal{G}$ since it corresponds to rotating the separatrices adjacent to an odd degree singularity by $2\pi$, and preserving all the other separatrices. Hence, we have $\phi_i((1,\dots ,1,\zeta,1,\dots ,1),Id)=-1=\zeta^{p_i}$. Hence $\phi_i((\zeta_1,\dots,\zeta_{n_i}),Id)=(\zeta_1\dots \zeta_{n_i})^{p_i}$.
The map $f_i:\tau_i\to \phi_i((1,\dots ,1), \tau_i)$ is a homomorphism hence is either the signature or the identity map.
Let $g_i$ be the element given by Lemma~\ref{lemme:geometrique}. A simple computation gives, $g_i^2=((\zeta_1\zeta_2,\zeta_1\zeta_2,1,\dots ,1),Id)$, hence $$g_i^{2p_i}=((\zeta_1^{p_i}\zeta_2^{p_i},\zeta_1^{p_i}\zeta_2^{p_i},1\dots ,1), Id)\neq 1$$
But $\phi_i(g_i)=(\zeta_1\zeta_2)^{p_i} f_i((1,2))$, hence $f_i((1,2))=-1$. 
\end{proof}

\begin{lemma}\label{lemme:geometrique}
For each $i$ such that $k_i$ is odd, there exists an element $g_i\in \mathcal{G}\cap \mathcal{G}_i$ of the form $((\zeta_1,\zeta_2,1,\dots ,1),(1,2))$ such that $g_i^2$ is of order $2p_i=k_i+1$.
\end{lemma}

\begin{proof}
We have defined in Section~\ref{link:to:transl} the moduli space of translation surfaces with labeled separatrices $\mathcal{H}^{lab}$. We also define the moduli space $\mathcal{H}^{sing}$ of translation surfaces with marked singularities (see \cite{B:labeled} for a precise definition). There are canonical coverings $\mathcal{H}^{lab}\to \mathcal{H}^{sing}$ and $\mathcal{H}^{sing}\to \mathcal{H}$. For any connected component $\mathcal{C}$ of $\mathcal{H}$, the preimage in $\mathcal{H}^{sing}$ is connected (once a proper condition on the label is fixed).

We start from a labeled permutation $\pi$, and choose a suspension datum $\zeta$ to obtain an element $S$ in $\mathcal{H}^{lab}(k,k_1^{n_1},\dots ,k_r^{n_r})$. 

We describe a path in $\mathcal{H}^{lab}(k,k_1^{n_1},\dots ,k_r^{n_r})$ that will induce the element $g_i$ required. Denote by $\mathcal{C}$ the underlying connected component of the (usual) moduli space of Abelian differentials.

From \cite{KoZo}, there exists a surface $S_0\in \mathcal{C}$ obtained after breaking up a singularity of order $2k_i$ into a pair of singularities of order $k_i$. We denote by ``$1$'' and ``$2$''  the singularities corresponding to the permutation in $g_i$. Seeing $S,S_0$ as elements in $\mathcal{H}^{sing}$, we can assume that the pairs of singularities  in $S_0$ after the breaking up procedure are ``1'' and ``2'', and there exists a path $\gamma$ joining $S$ to $S_0$. With the notations of  Section~\ref{def:breaking}, we cut the surface along $\partial V_\varepsilon$, rotate   the disc by an angle $\theta$, one gets a family of surfaces $(S_\theta)$.  One has $S_{\pi(2k_i+1)}=S_0$ in $\mathcal{H}$, and composing with $\gamma^{-1}$, we obtain a closed path in $\mathcal{C}$. Considering the lift of this path in $\mathcal{H}^{lab}$ we see that:
\begin{itemize}
\item The two singularities ``1'' and ``2'' have been interchanged.
\item All the other singularities, and the corresponding labels on horizontal separatrice are fixed, since the surgery (cutting on a circle, rotating, pasting) does not change the metric outside a neighborhood of the two singularities.
\end{itemize}
Hence the resulting element $g_i$ in $Aut(\mathcal{D})$ is in $\mathcal{G}_i'$ and of the form $((\zeta_1,\zeta_2,1,\dots ,1),(1,2))$.


Now we look at $g_i^2$. It corresponds to the following path: consider $\gamma$, then the path $(S_\theta)$, for $\theta\in [0,\dots 2\pi(2k_i+1)]$, then $\gamma^{-1}$.  Keeping track of the marked horizontal separatrices, we see at the end that the marked horizontal separatrices for ``1'', ``2'' have changed by an angle $-2\pi(2k_i+1) \mod 2\pi(k_i+1)$, hence $2\pi$. So we have: 

$$g_i^2=((\zeta,\zeta,1,\dots ,1),Id).$$

With $\zeta=exp(\frac{2i\pi}{k_i+1})$, hence $g_i^2$ is of order $k_i+1$.

\end{proof}

\nocite*
\bibliographystyle{plain}
\bibliography{biblio}

\end{document}